\documentclass[nonblindrev]{informs3}

\OneAndAHalfSpacedXI 

\usepackage{endnotes}
\let\footnote=\endnote

\usepackage{natbib}
 \bibpunct[, ]{(}{)}{,}{a}{}{,}%
 \def\bibfont{\small}%

 \usepackage{graphicx}
\usepackage{color}
\usepackage{amsmath,amstext,amssymb}
\usepackage{mathrsfs}
\usepackage{multirow}

\TheoremsNumberedThrough     
\ECRepeatTheorems

\EquationsNumberedThrough    

\newcommand{\trace}{{\mbox{\textrm{Tr}}}}
\newcommand{\rank}{{\mbox{\textrm{Rank}}}}

\newcommand{\bfalpha}{{\mbox{\boldmath $\alpha$}}}
\newcommand{\bfxi}{{\mbox{\boldmath $\xi$}}}

\newcommand{\bfbeta}{{\mbox{\boldmath $\beta$}}}

\newcommand{\st}{{\rm s.t.}}
\newcommand{\Prob}{{\rm Prob}}

\newcommand{\by}{\mathbf{y}}

\newcommand{\bx}{\mathbf{x}}

\newcommand{\bH}{\mathbf{H}}

\newcommand{\bX}{\mathbf{X}}

\newcommand{\bw}{\mathbf{w}}

\newcommand{\bh}{\mathbf{h}}

\newcommand{\bhX}{\widehat{\mathbf{X}}}

\newcommand{\cM}{\mathcal{M}}
\newcommand{\cN}{\mathcal{N}}

\begin{document}


\RUNAUTHOR{Z. Xu and M. Hong}

\RUNTITLE{Semidefinite Approximation for Two MBQCQPs}

\TITLE{Semidefinite Relaxation for Two Mixed Binary Quadratically Constrained Quadratic Programs: Algorithms and Approximation Bounds}

\ARTICLEAUTHORS{%
\AUTHOR{Zi Xu}
\AFF{Department of Mathematics, College of Sciences, Shanghai University, Shanghai, 200444, P. R. China, \EMAIL{xuzi@shu.edu.cn.}} 
\AUTHOR{Mingyi Hong}
\AFF{Department of Electrical and Computer Engineering, University of Minnesota, Minneapolis, MN 55455, \EMAIL{mhong@umn.edu.}}
} 

\ABSTRACT{%
This paper develops new semidefinite programming (SDP) relaxation techniques for two classes of mixed binary quadratically constrained quadratic programs (MBQCQP) and analyzes their
approximation performance. The first class of problem finds two minimum norm vectors in $N$-dimensional real or complex Euclidean space, such that $M$ out of $2M$ concave quadratic functions are satisfied. By employing a special randomized rounding procedure, we show that the ratio between the norm of the optimal solution of this model and its SDP relaxation is upper bounded by $\frac{54M^2}{\pi}$ in the real case and by $\frac{24M}{\sqrt{\pi}}$ in the complex case. The second class of problem finds a series of minimum norm vectors
subject to a set of quadratic constraints and a cardinality constraint with both binary and continuous variables. We show that in this case the approximation ratio is also bounded and independent of problem dimension for both the real and the complex cases.
}%


\KEYWORDS{nonconvex quadratically constrained quadratic programming; semidefinite program relaxation; approximation bound; NP-hard} 

\maketitle

%


\section{Introduction}
In this paper, we study  two classes of mixed binary nonconvex quadratically constrained quadratic programming (MBQCQP) problems, where the objective functions are quadratic in the continuous variables and the constraints contain both continuous and binary variables. These two classes of optimization problems are difficult as they are nonconvex even with the binary variables being fixed.  The focus of our study is to design efficient semidefinite programming (SDP) based algorithms for both problems and to analyze their approximation bounds.

{\bf The first model.}
Consider the following MBQCQP problem:
\begin{align}
\min_{\bw_1, \bw_2\in \mathbb{F}^{N}, \mbox{\footnotesize\boldmath$\beta$}}&\quad\|\bw_1\|^2+\|\bw_2\|^2\nonumber\\
{\st }&\quad \bw_1^H\bH_i\bw_1\ge \beta_i,\ i=1\cdots,M\tag{P1}\label{problem},\nonumber\\
&\quad \bw_2^H\bH_i\bw_2\ge 1-\beta_i,\ i=1\cdots,M,\nonumber\\
&\quad \beta_i\in\{0,1\},\ i=1,\cdots,M\nonumber,
\end{align}
where $\mathbb{F}$ is either the field of real numbers $\mathbb{R}$ or the field of complex numbers $\mathbb{C}$; $\bfbeta:=(\beta_1, \cdots, \beta_M)^T$; $\bH_i:=\bh_i\bh_i^H$ $(i=1,\cdots, M)$ and $\bh_i$'s are $N$ dimensional real or complex vectors; $\|\cdot\|$ denotes the Euclidean norm in $\mathbb{F}^N$ and $M>0$ is a given integer.
Throughout, we use the superscript $H$ to denote the complex Hermitian transpose when $\mathbb{F}=\mathbb{C}$, and use it to denote transpose if $\mathbb{F}=\mathbb{R}$.  Notice that the problem \eqref{problem} can be easily solved either when $N=1$ or $M=1$, by solving a maximum eigenvalue problem. Hence, we shall assume that $N\ge 2$ and $M\ge 2$ in the rest of the paper.
Let $v^{\rm P1}$ denote its optimal objective value.
We note that problem \eqref{problem} is in general NP-hard. The reason is that even if we fix all the binary variables, the problem is eqivalent to two subproblems, which are both NP-hard; see \cite[Section
2]{Luo07approximationbounds}.

Our interest in problem \eqref{problem} is motivated by its
application in telecommunications. For example,
consider a cellular network where $M$ users, each equipped with a single
receive antenna, are served by a base station (BS) with $N$ transmit antennas.
Let $\bar{\mathbf{h}}_i\in\mathbb{C}^{N\times 1}$ denote the complex
channel coefficient between the BS and user $i$, and let $n_i$
denote the thermal noise power at the receiver of user $i$. Let
$\cM=\{1,\cdots,M\}$ denote the set of all users. Assume that there are two time slots, and for each slot $q=\{1,2\}$ a linear transmit beam $\mathbf{w}_q\in\mathbb{C}^{N\times 1}$ is
used by the BS to transmit a common message to the users.
Using these notations, the signal to noise ratio (SNR) at the receiver of
each user $i\in\mathcal{M}$ at a give slot can be expressed as ${\rm
SNR}_{i}\triangleq\frac{|\bw_q^H\bar{\bh}_i|^2}{n_i}$.
In order to successfully decode the transmitted message, typically a
quality of service (QoS) requirement in the form of
$\frac{|\bw_q^H\bar{\bh}_i|^2}{n_i}\ge\gamma_i$ is imposed by each
user $i\in\cM$, where $\gamma_i$ is a predetermined QoS threshold.
Let us define $\bh_i:=\frac{\bar{\bh}_i}{\sqrt{n_i
\gamma_i}}$ as user $i$'s normalized channel.

When the number of users in the network is large, proper user scheduling algorithm is needed in order
to guarantee the QoS for all the users. More specifically, the objective of the design is threefold: 1)
Properly assign the users into different time slots; 2) Multicast the desired signal to {\it each} time slot; 3) Minimize the total transmit power across all slots. In the simplest case where there are only two slots available, the problem can be formulated as the MBQCQP probem \eqref{problem} by using $\beta_i\in\{0,1\}$ and $1-\beta_i\in\{0,1\}$ to represent whether user $i$ is scheduled in slot $1$ or slot $2$. This problem is called joint user grouping and physical layer transmit beamforming.

In the special case where there is only a single slot available, the physical layer multicast problem tries to minimize the total transmit power while satisfying {\it all} the users' QoS constraint. Mathematically it is equivalent to problem \eqref{problem} with $\beta_i=1$ for all $i$.
The resulting problem is a continuous QCQP problem, whose SDP relaxation bounds have been extensively studied; see e.g., \cite{Luo07approximationbounds}. When applying SDP relaxation to solve the related problem, the first step is to reformulate the problem by introducing a rank-1 matrix $\bX=\bw\bw^H$. After
dropping the nonconvex rank-1 constraint on $\bX$, the relaxed
problem becomes an SDP, whose optimal solution ${\bar \bX}$ can be
efficiently computed. A randomization procedure then follows which
converts ${\bar \bX}$ to a feasible solution. It
has been shown in \cite{Luo07approximationbounds} that such SDP
relaxation scheme generates high-quality solutions, and the theoretical ratio between its optimal solution and the SDP relaxation has been shown
to be upper bounded by $27M^2/\pi$ \cite{Luo07approximationbounds}.
However, for more general case \eqref{problem} with both binary and continuous variables,
there is no known performance guarantees for the performance of SDP
relaxation techniques.

{\bf The second model.} Another interesting case of the MBQCQP
problem is the generalization for the first model. In this more general model, we consider an alternative formulation which can deal with the case with multiple time slots. Consider the same system model as described in the first model, except that we have a total of $Q$ time slots available for scheduling. Let
$\bfbeta:=[\bfbeta^{(1)}, \cdots,\bfbeta^{(Q)}]$, where
$\bfbeta^{(q)}\in\{0,1\}^{M}$. Let $\bw:=[(\bw_1)^H,\cdots,
(\bw_Q)^H]^H$, where $\bw_q$ represents the beamformer used in the $q$th
time slot. Consider the following problem
\begin{align}
\min_{\bw_q, \mbox{\footnotesize\boldmath$\beta$}^{(q)}}&\quad\sum_{q=1}^{Q}\|\bw_q\|^2\nonumber\\
{\st}&\quad \bw_q^H\bH_i\bw_q\ge \beta^{(q)}_i,\ i=1\cdots,M,\ q=1\cdots, Q\tag{P2}\label{problemQReformulate}\\
&\quad\sum_{q=1}^{Q}\beta_i^{(q)}\ge P_i, \ i=1,\cdots,M\nonumber\\
&\quad \beta^{(q)}_i\in\{0,1\},\ \ i=1,\cdots,M,\ q=1,\cdots
Q\nonumber.
\end{align}
Note that we have generalized the problem to the one that allows
each user $i$ to be assigned to {\it at least $P_i$ time slots}.
Note that $P_i\in[1,\cdots, Q]$ is a given integer that can
represent the users's service priorities. The higher the value of
$P_i$, the larger the number of slots will be reserved for user $i$.
 let $V^{\rm P2}$ denote its optimal objective value. Note that the above formulation, even for the $Q=2$ case, we have $M$ additional discrete integer variables compared with our previous formulation.
For this more general case, no approximation bounds for SDP relaxation are known either.

In the absence of the discrete constraints, there is an extensive literature on the quality bounds of SDP relaxation for solving nonconvex QCQP problems, in either a maximization or a minimization form
\cite{Luo07approximationbounds,Nemirovski99,Bental02,Beck, He2008}. In the seminal work by \cite{Goemans:1995}, the authors show that for the max-cut problem, which can be formulated as certain QCQP problem with discrete variables only, the ratio of the optimal value of SDP relaxation over that of the original problem is bounded below by $0.87856\ldots$. Other related results can be found in \cite{ye01_699, Frieze:1995}. Some alternative approaches to
MBQCQP have recently appeared in \cite{Billionnet2009, Burer2011, Saxena10,
Saxena11, Billionnet2012}. More detailed reviews of recent progress on related problems
can be found in the excellent surveys \cite{Burer2012, Hemmecke, Koppe}.

Although SDP relaxation technique has been quite successful in solving continuous QCQP, so far little is known about the effectiveness of applying it for MBQCQP, except for two recent results \cite{Hong2013, Xu2013}. In those two papers, the authors study MBQCQP models that differ substantially from those considered in this work. In \cite{Hong2013}, the authors considered a quadratic maximization problem with a {\it single} cardinality constraint on the binary variables, and those binary variables are used to {\it reduce the dimensionality of continuous variables}. In \cite{Xu2013}, the authors consider an MBQCQP problem where the objective is to find a minimum norm vector under some concave mixed binary quadratic constraints and a  single cardinality constraint on the binary variables. A special case of that problem is to
 find a minimum norm vector while satisfying $Q$ out of $M$ quadratic constraints. In this case the binary variables are used to {\it reduce the number of quadratic constraints}. In the current work, we consider a significantly harder problem where there are {\it multiple} cardinality constraints for the binary variables. As a result, the techniques and analysis developed in \cite{Hong2013, Xu2013} can not be applied directly. We note that if we set $M=1$, $P_1=q$ and $Q=m$, then problem \eqref{problemQReformulate} reduces to a special case of the minimization model considered in \cite{Xu2013} with $\epsilon=0$.

{\bf Our Contributions.} In Sections $2$, we develop novel SDP relaxation techniques for solving the models \eqref{problem} and \eqref{problemQReformulate}. The main idea is to first relax the binary variables to continuous variables and using the SDP relaxation for the rest of the continuous variables. Given an optimal solution of the relaxed problem, we devise new randomization procedures to generate approximate solutions for the original NP-hard MBQCQP problems. Moreover, we analyze the quality of such approximate solutions by
deriving  bounds on the approximation ratios between the optimal solution of
the two MBQCQP problems and those of their corresponding SDP relaxations for both the real and complex cases.

{\bf Notations.} For a symmetric matrix $\mathbf{X}$,
$\mathbf{X}\succeq 0$ signifies that $\mathbf{X}$ is positive
semi-definite. We use $\trace\left[\mathbf{X}\right]$ and $\bX[i,j]$
to denote the trace and the $(i,j)$th element of a matrix $\bX$, respectively.
For a vector $\bx$, we use $\|\bx\|$ to denote its Euclidean norm,
and use $\bx[i]$ to denote its $i$th element. For a real vector
$\by$, use $\by_{[m]}$ to denote its $m$th largest elements. For a
complex scalar $x$, its complex conjugate is denoted by $\bar{x}$.
Given a set $\mathcal{A}$, $|\mathcal{A}|$ denotes the number of
elements in set $\mathcal{A}$. Also, we use $\mathbb{R}^{N\times M}$ and
$\mathbb{C}^{N\times M}$ to denote the set of real and complex
$N\times M$ matrices, and use $\mathbb{S}^{N}$ and
$\mathbb{S}^{N}_{+}$ to denote the set of $N\times N$ hermitian and
hermitian positive semi-definite matrices, respectively.
Finally, we use the superscript $H$ and $T$ to denote the complex Hermitian transpose and transpose of a matrix or a vector respectively.

\section{Approximation Algorithm and Approximation Ratio Analysis}
\subsection{The Proposed Algorithm for \eqref{problem} and Its Approximation Ratio}\label{secTwoSlots}
We consider a relaxation of
the problem \eqref{problem}, expressed as follows
\begin{align}
\min_{\bX^{(1)}, \bX^{(2)}, \bfalpha}&\quad\trace[\bX^{(1)}]+\trace[\bX^{(2)}]\nonumber\\
{\st}&\quad\trace[\bH_{i}\bX^{(1)}]\ge \alpha_i, i=1,\cdots,M\tag{SDP1}\label{prblemSDR}\\
&\quad\trace[\bH_{i}\bX^{(2)}]\ge1-\alpha_i, i=1,\cdots,M\nonumber\\
&\quad 0\leq \alpha_i \leq 1, i=1,\cdots,M\nonumber\\
&\quad \bX^{(1)}\succeq 0, \bX^{(2)}\succeq 0,\nonumber
\end{align}
where we do SDP relaxation for the continuous variables and continuous relaxation for the binary variables.  Let $(\bX^{*(1)}, \bX^{*(2)}, \bfalpha^*)$ denote the optimal
solution to this SDP problem, and let $v^{\rm SDP1}$ denote
its optimal objective value.

In the following, we aim to generate a feasible solution  $(\bar\bx^{(1)}, \bar\bx^{(2)}, \bar\bfalpha^*)$ for
\eqref{problem} from  $(\bX^{*(1)}, \bX^{*(2)}, \bfalpha^*)$ , and evaluate the quality of
such solution. In particular, we would like to find a constant
$\mu\ge 1$ such that
$$\|\bar\bx^{(1)}\|^2+\|\bar\bx^{(2)}\|^2 \le \mu v^{\rm SDP1}.$$
By using the fact that such generated solution is feasible for
problem \eqref{problem}, we have $v^{\rm P1}\le \|\bar\bx^{(1)}\|^2+\|\bar\bx^{(2)}\|^2$, which further implies that the same $\mu$ is an
upper bound of the SDP relaxation performance, i.e.,
\begin{align}
v^{\rm P1}\le \mu v^{\rm SDP1}\label{defratio}.
\end{align}
The constant $\mu$ will be referred to as the {\it approximation
ratio}.

Once we obtain the solution $(\bX^{*(1)}, \bX^{*(2)}, \bfalpha^*)$, we first need to convert it into a
feasible solution for the problem \eqref{problem}. For such purpose, we propose a new randomization procedure; see Table \ref{tableRandomize}.

\begin{table}[htb]
\begin{center}
\vspace{-0.1cm} \caption{The Randomization Procedure}
\label{tableRandomize} {\small
\begin{tabular}{|l|}
\hline
\\
S0:~ Let $\bar\bfalpha^*_i=1$ if $\bfalpha^*_i\geq \frac{1}{2}$, otherwise let $\bar\bfalpha^*_i=0$.\\
\quad \quad Denote $\mathcal{I}=\{i \mid \bar\bfalpha^*_i=1, i=1, \cdots, M\}$.\\

S1: ~Sample $\bfxi^{(1)}\sim \cN_c(\bf{0}, \bX^{(1)*})$; \\
\quad \quad Sample $\bfxi^{(2)}\sim \cN_c(\bf{0}, \bX^{(2)*})$; \\

S2: ~a) Let ${t^{(1)}}=\max_{i\in \mathcal{I}}\bigg\{\sqrt{\frac{1}{\bfxi^{(1)H}\bH_i\bfxi^{(1)}}}\bigg\}$;\\

\quad \quad b) Let  ${t^{(2)}}=\max_{i\in \mathcal{M}\setminus\mathcal{I}}\bigg\{\sqrt{\frac{1}{\bfxi^{(2)H}\bH_i\bfxi^{(2)}}}\bigg\}$;\\

S4: Let $\bar\bx^{(1)}={t^{(1)}}\bfxi^{(1)}$, $\bar\bx^{(2)}={t^{(2)}}\bfxi^{(2)}$.\\

\\
  \hline
\end{tabular}}
\vspace{-0.3cm}
\end{center}
\end{table}

Note that after a single execution of the algorithm, we obtain a feasible solution $(\bar\bx^{(1)}, \bar\bx^{(2)}, \bar\bfalpha^*)$ for the original problem \eqref{problem},  and we have
$$v^{P1}(\bar\bx^{(1)}, \bar\bx^{(2)})=(t^{(1)})^2\|\bfxi^{(1)}\|^2+(t^{(2)})^2\|\bfxi^{(2)}\|^2.$$
In the following two theorems, we will show the approximation ratio for our algorithm for both real and complex cases.

\begin{theorem} \label{thm:2.1}
There exists a constant $\sigma>0$ such that
\begin{align}
P\left(v^{\rm P1}(\bar\bx^{(1)}, \bar\bx^{(2)})\le\mu v^{\rm SDP1}\right)\ge \sigma,\label{prob:firstmodel}
\end{align}
with $\mu=\frac{54M^2}{\pi}$ when  $\mathbb{F}=\mathbb{R}$ and $\mu=\frac{24M}{\sqrt{\pi}}$ when  $\mathbb{F}=\mathbb{C}$.
\end{theorem}

\begin{proof}
Let us pick any $\alpha>0$, and we have the following series of inequalities
\begin{align}
&P\left(v^{\rm P1}(\bar\bx^{(1)}, \bar\bx^{(2)})\le\mu v^{\rm SDP1}\right)\nonumber\\
&=P\left(\|\bar\bx^{(1)}\|^2+\|\bar\bx^{(2)}\|^2\le
\mu(\trace[\bX^{*(1)}]+\trace[\bX^{*(2)}])\right)\nonumber\\
&\ge P\left(\|\bfxi^{(1)}\|^2+\|\bfxi^{(2)}\|^2\le
\frac{\mu}{\alpha}(\trace[\bX^{*(1)}+\trace[\bX^{*(2)})], (t^{(1)})^2\le \alpha, (t^{(2)})^2\le \alpha \right)\nonumber\\
&\ge 1-P\left((t^{(1)})^2> \alpha\right)-P\left( (t^{(2)})^2>\alpha\right)-
P\left(\|\bfxi^{(1)}\|^2+\|\bfxi^{(2)}\|^2>
\frac{\mu}{\alpha}(\trace[\bX^{*(1)}+\trace[\bX^{*(2)})]\right)\nonumber\\
&= P\left((t^{(1)})^2\le \alpha\right)+P\left( (t^{(2)})^2\le
\alpha\right)-1-P\left(\|\bfxi^{(1)}\|^2+\|\bfxi^{(2)}\|^2>
\frac{\mu}{\alpha}(\trace[\bX^{*(1)}+\trace[\bX^{*(2)})]\right)\nonumber\\
& \ge P\left(\frac{1}{\left(\bfxi^{(1)}\right)^H\bH_i\bfxi^{(1)}}\le \alpha, \forall\;i\in \mathcal{I}\right)+ P\left(\frac{1}{\left(\bfxi^{(2)}\right)^H\bH_i\bfxi^{(2)}}\le \alpha, \forall\; i\in\mathcal{M}\setminus\mathcal{I}\right)\nonumber\\
 &\quad -1-\frac{\alpha}{\mu}.
\end{align}
where the last inequality is from the Markov's inequality.  Since
\begin{align}
&\trace\left[\bH_i\bX^{*(1)}\right]\ge \alpha_i^*\ge \frac{1}{2},  \forall\; i \in \mathcal{I},\\
&\trace\left[\bH_i\bX^{*(2)}\right]\ge 1-\alpha_i^*\ge \frac{1}{2},  \forall\; i \in \mathcal{M}\setminus\mathcal{I},
\end{align}
we have
\begin{align}
&P\left(v^{\rm P1}(\bar\bx^{(1)}, \bar\bx^{(2)})\le\mu v^{\rm SDP1}\right)\nonumber\\
&\ge P\left(\left(\bfxi^{(1)}\right)^H\bH_i\bfxi^{(1)}\ge \frac{2\trace[\bH_i\bX^{*(1)}]}{\alpha}, \forall\; i\in \mathcal{I}\right)\nonumber\\
&\quad + P\left(\left(\bfxi^{(2)}\right)^H\bH_i\bfxi^{(2)}\ge \frac{2\trace[\bH_i\bX^{*(2)}]}{\alpha}, \forall\; i\in\mathcal{M}\setminus\mathcal{I}\right) -1-\frac{\alpha}{\mu}\nonumber\\
&= 1-  P\left(\left(\bfxi^{(1)}\right)^H\bH_i\bfxi^{(1)}\le \frac{2\trace[\bH_i\bX^{*(1)}]}{\alpha}, ~\mbox{for some}~ i\in \mathcal{I}\right)\nonumber\\
&\quad -  P\left(\left(\bfxi^{(2)}\right)^H\bH_i\bfxi^{(2)}\le \frac{2\trace[\bH_i\bX^{*(2)}]}{\alpha},~\mbox{for some}~ i\in\mathcal{M}\setminus\mathcal{I}\right)-\frac{\alpha}{\mu}.\label{probbound}
\end{align}
When $\mathbb{F}=\mathbb{R}$, then by utilizing \cite[Lemma 1]{Luo07approximationbounds}, and from \eqref{probbound}, we obtain
\begin{align}
P\left(v^{\rm P1}(\bar\bx^{(1)}, \bar\bx^{(2)})\le\mu v^{\rm SDP1}\right)
\ge 1-\sum_{i=1}^{M}\max\left\{\sqrt{\frac{2}{\alpha}}, \frac{2(\bar{r}_i-1)}{\pi-2}\cdot\frac{2}{\alpha}\right\}-\frac{\alpha}{\mu}.
\end{align}
Since $\bH_i=\bh_i\bh_i^H$, we have $\bar{r}_i=1$ for all $i=1,\cdots, M$ and thus
\begin{align}
P\left(v^{\rm P1}(\bar\bx^{(1)}, \bar\bx^{(2)})\le\mu v^{\rm SDP1}\right) &\ge 1-M\cdot\sqrt{\frac{2}{\alpha}}-\frac{\alpha}{\mu}.
\end{align}
By choosing $\alpha=\frac{18M^2}{\pi}$, and set $\mu=3\alpha$ and $\sigma=\frac{2-\sqrt{\pi}}{3}$,
we have $$P\left(v^{\rm P1}(\bar\bx^{(1)}, \bar\bx^{(2)})\le\mu v^{\rm SDP1}\right)\ge \sigma>0.$$
When $\mathbb{F}=\mathbb{C}$, then by utilizing \cite[Lemma 3]{Luo07approximationbounds}, and from \eqref{probbound}, we obtain
\begin{align}
P\left(v^{\rm P1}(\bar\bx^{(1)}, \bar\bx^{(2)})\le\mu v^{\rm SDP1}\right)
\ge 1-\sum_{i=1}^{M}\max\left\{\frac{4}{3}\cdot\frac{2}{\alpha}, 16(\bar{r}_i-1)^2\cdot\frac{4}{\alpha^2}\right\}-\frac{\alpha}{\mu}.
\end{align}
Since $\bH_i=\bh_i\bh_i^H$, we have $\bar{r}_i=1$ for all $i=1,\cdots, M$ and thus
\begin{align}
P\left(v^{\rm P1}(\bar\bx^{(1)}, \bar\bx^{(2)})\le\mu v^{\rm SDP1}\right) &\ge 1-M\cdot\frac{8}{3\alpha}-\frac{\alpha}{\mu}.
\end{align}
By choosing $\alpha=\frac{8M}{\sqrt{\pi}}$, and set $\mu=3\alpha$ and $\sigma=\frac{2-\sqrt{\pi}}{3}$,
we have $$P\left(v^{\rm P1}(\bar\bx^{(1)}, \bar\bx^{(2)})\le\mu v^{\rm SDP1}\right)\ge \sigma>0.$$
This completes the proof.
\end{proof}

\begin{theorem} \label{thm:2.2}
For \eqref{problem} and its SDP relaxation \eqref{prblemSDR}, we have
\begin{align}
v^{\rm P1}\le\mu v^{\rm SDP1},
\end{align}
with $\mu=\frac{54M^2}{\pi}$ when  $\mathbb{F}=\mathbb{R}$ and $\mu=\frac{24M}{\sqrt{\pi}}$ when  $\mathbb{F}=\mathbb{C}$.
\end{theorem}
\begin{proof}
We see from the inequality \eqref{prob:firstmodel} that there is a positive
probability (independent of problem size) of at least
$$\frac{2-\sqrt{\pi}}{3}=0.0758\ldots$$ that
$$\max\left\{\max_{i\in \mathcal{I}}\bigg\{\sqrt{\frac{1}{\bfxi^{(1)H}\bH_i\bfxi^{(1)}}}\bigg\},\max_{i\in \mathcal{M}\setminus\mathcal{I}}\bigg\{\sqrt{\frac{1}{\bfxi^{(2)H}\bH_i\bfxi^{(2)}}}\bigg\}\right\}\leq \alpha,$$ with $\alpha=\frac{8M}{\sqrt{\pi}}$ and
$$
\|\bfxi^{(1)}\|^2+\|\bfxi^{(2)}\|^2\leq 3(\trace[{\bX}^{*(1)}]+\trace[{\bX}^{*(2)}]).$$
Let $\bfxi^{(1)}$ and $\bfxi^{(2)}$ be any
vectors satisfying these two conditions\footnote{The probability that no such $\bfxi^{(1)}$ or $\bfxi^{(2)}$ are generated after $N$ independent trials is at most $(1-0.0758\ldots)^N$,
which for $N=100$ equals $0.000375\ldots$. Thus, such $\bfxi^{(1)}$ or $\bfxi^{(2)}$ require relatively few trials to generate, see also in \cite{Luo07approximationbounds}.
}. Then $\bx$ is feasible for
\eqref{problem}, so that
\begin{align}
v^{{\rm P1}}&\le \|\bar\bx^{(1)}\|^2+\|\bar\bx^{(2)}\|^2\nonumber\\
&=(\|\bfxi^{(1)}\|^2+\|\bfxi^{(2)}\|^2)\cdot \max\left\{\max_{i\in \mathcal{I}}\bigg\{\sqrt{\frac{1}{\bfxi^{(1)H}\bH_i\bfxi^{(1)}}}\bigg\},\max_{i\in \mathcal{M}\setminus\mathcal{I}}\bigg\{\sqrt{\frac{1}{\bfxi^{(2)H}\bH_i\bfxi^{(2)}}}\bigg\}\right\}\nonumber\\
&\leq \alpha\cdot 3(\trace[{\bX}^{*(1)}]+\trace[{\bX}^{*(2)}])=\mu \cdot  v^{{\rm SDP1}},
\end{align}
where the last equality uses $\trace[{\bX}^{*(1)}]+\trace[{\bX}^{*(2)}]=v^{{\rm SDP1}}$ and $\mu$ is defined as in Theorem \ref{thm:2.1}.
\end{proof}

\subsection{The Proposed Algorithm for \eqref{problemQReformulate} and Its Approximation Ratio}
We consider a relaxation of problem
\eqref{problemQReformulate}, expressed as follows:
\begin{align}
\min_{\bX^{(q)},  \bfalpha^{(q)}}&\quad \sum_{q=1}^{Q}\trace[\bX^{(q)}]\nonumber\\
{\st}&\quad\trace[\bH_{i}\bX^{(q)}]\ge \alpha_i^{(q)}, i=1,\cdots,M,\ q=1\cdots, Q,\nonumber\\
&\quad  \sum_{q=1}^{Q}\alpha_i^{(q)}\ge P_i, \ i=1,\cdots,M,\ q=1\cdots, Q,\tag{SDP2}\label{problemSDR2}\\
&\quad 0\leq \alpha^{(q)}_i \leq 1, i=1,\cdots,M,\ q=1\cdots, Q,\nonumber\\
&\quad \bX^{(q)}\succeq 0, \ q=1\cdots, Q,\nonumber
\end{align}
where we do SDP relaxation for the continuous variables and continuous relaxation for the binary variables. Let $\bhX^{(q)}$ and $\hat{\bfalpha}^{(q)}=(\hat{\alpha}_1^{(q)}, \hat{\alpha}_2^{(q)}, \cdots, \hat{\alpha}_M^{(q)})^T$, $q=1,\cdots,Q$ denote the optimal
solution to this SDP problem,  and let $v^{\rm
SDP2}=\sum_{q=1}^{Q}\trace[\bhX^{(q)}]$ denote its optimal
objective value.

We propose the algorithm listed in Table \ref{tableRandomize_G} to obtain a feasible solution
$\{(\tilde\bx^{(q)}, \tilde\bfalpha^{(q)})\}_{q=1,\cdots,Q}$ to the problem \eqref{problemQReformulate}. The main idea of the algorithm is given below. For each user $i$, we first construct a set $\mathcal{P}_i$ that collects the largest $P_i$ components from $\{\hat{\alpha}^{(1)}_i, \cdots, \hat{\alpha}^{(q)}_i\}$. The binary variables are then determined using these sets $\mathcal{P}_i$, $i=1,\cdots, M$. Once the binary variables are fixed, standard randomization technique is use to construct each continuous variable $\tilde{\mathbf{x}}^{(q)}$.

\begin{table}[htb]
\begin{center}
\vspace{-0.1cm} \caption{ The Randomization Procedure}
\label{tableRandomize_G} {\small
\begin{tabular}{|l|}
\hline
\\
S0: Denote $\mathcal{Q}:=\{1, \dots, Q\}$ and define index sets:\\
\quad\quad $\mathcal{P}_i:=\{j \mid \hat{\alpha}_i^{(j)}\ge \hat{\alpha}_i^{[P_i]}, j=1, \cdots, Q\}, \ i=1,\cdots, M,$\\
\quad \quad where $\hat{\alpha}_i^{[P_i]}$ is the $P_i$-th largest element in vector $(\hat{\alpha}_i^{(1)}, \hat{\alpha}_i^{(2)}, \cdots, \hat{\alpha}_i^{(Q)})$. \\

S1: For $i=1, \cdots, M$, \\
      \quad \quad set $\tilde\alpha_i^{(q)}=1$ for all $q\in \mathcal{P}_i$ and $i=1, \cdots, M$;\\
      \quad \quad otherwise set $\tilde\alpha_i^{(q)}=0$ for all $q\in \mathcal{Q}\setminus\mathcal{P}_i$ and $i=1, \cdots, M$.\\
\quad \quad Let $\tilde\bfalpha^{(q)}=(\tilde\alpha_1^{(q)}, \tilde\alpha_2^{(q)}, \cdots, \tilde\alpha_M^{(q)})^T$ for all $q=1,\cdots,Q$.\\

S2: For all $q\in  \mathcal{Q}$, denote\\
\quad \quad $\mathcal{S}^{(q)}=\{i \mid \tilde\alpha_i^{(q)}=1, \ i=1, \cdots, M\}$.\\

S3: Generate a random vector $\bfxi^{(q)}$ from the Normal distribution $\cN_c({\bf{0}}, \bhX^{(q)})$, $q=1,\cdots, Q$. \\

S4: Let $\tilde\bx^{(q)}=\tilde{t}^{(q)}\bfxi^{(q)}$, with\\
\quad \quad
${\tilde{t}}^{(q)}=\sqrt{\max_{i\in\mathcal{S}^{(q)}}
\bigg\{\frac{1}{\bfxi^{(q)H}\bH_i\bfxi^{(q)}}\bigg\}}$, for all $q=1, \cdots, Q$.\\
\\
  \hline
\end{tabular}}
\vspace{-0.3cm}
\end{center}
\end{table}


\begin{lemma}\label{claimXLowerBound}
{\it For all $i=1, \cdots, M$, if $q\in \mathcal{P}_i$, then $\hat{\alpha}_i^{(q)}\ge \frac{1}{Q-P_i+1}$}.
\end{lemma}

\begin{proof}
First for any given $i$ in the set $\{1, \cdots, M\}$, by the feasibility of $\hat{\alpha}_i^{(q)}$, we have that
\begin{align}
\sum_{q=1}^{Q}\hat{\alpha}_i^{(q)}\ge P_i.\label{eqSumXLowerBound}
\end{align}
We claim that at least $P_i$ elements of the vector $\{\hat{\alpha}_i^{(1)}, \hat{\alpha}_i^{(2)}, \cdots, \hat{\alpha}_i^{(Q)}\}$ is greater than or equal to $\frac{1}{Q-P_i+1}$. We prove this claim by contradiction. Suppose that at most $P_i-1$ elements in this set is greater than or
equal to $\frac{1}{Q-P_i+1}$. Then we have
\begin{align}
&\sum_{q=1}^{Q}\hat{\alpha}_i^{(q)}<(Q-(P_i-1))\left(\frac{1}{Q-P_i+1}\right)+(P_i-1)=P_i.
\end{align}
This contradicts \eqref{eqSumXLowerBound}. The claim is proved. Since $\mathcal{P}_i$ contains the $P_i$ largest elements of the vector $\{\hat{\alpha}_i^{(1)}, \hat{\alpha}_i^{(2)}, \cdots, \hat{\alpha}_i^{(Q)}\}$, this completes the proof.
\end{proof}

Next, we analyze the above proposed algorithm.

\begin{theorem}\label{claimLowerBoundProbability}
{\it There exists a positive constant $\sigma_2$ such that
\[\Prob\left(\sum_{q=1}^{Q}\|\tilde\bx^{q}\|^2\le \bar\mu \left(\sum_{q=1}^{Q}\mbox{\trace}[\bhX^{(q)}]\right)\right)\ge
\sigma_2>0.\]
with
$$\bar\mu=\frac{27(\sum_{i=1}^{M}P_i\sqrt{Q-P_i+1})^2}{\pi},~ \mbox{when}~ \mathbb{F}=\mathbb{R};$$
and
$$\bar\mu=\frac{12\sum_{i=1}^{M}P_i(Q-P_i+1)}{\sqrt{\pi}},~ \mbox{when}~ \mathbb{F}=\mathbb{C}.$$}
\end{theorem}

\begin{proof}
The desired probability can be bounded below as follows
\begin{align}
&\quad\Prob\left(v^{\rm
P2} \le \bar\mu\  v^{\rm
SDP2} \right)\nonumber\\
&\ge \Prob\left(\sum_{q=1}^{Q}\|\tilde\bx^{q}\|^2\le \bar\mu \left(\sum_{q=1}^{Q}\trace[\bhX^{(q)}]\right)\right)\nonumber\\
&\ge \Prob\left(\sum_{q=1}^{Q}(\tilde{t}^{(q)})^2\|\bfxi^{(q)}\|^2\le
\bar\mu\left(\sum_{q=1}^{Q}\trace[\bhX^{(q)}]\right)\right)\nonumber\\
&\ge \Prob\left(\sum_{q=1}^{Q}\|\bfxi^{(q)}\|^2\le
\frac{\bar\mu}{\alpha_2}\left(\sum_{q=1}^{Q}\trace[\bhX^{(q)}]\right), \frac{1}{\bfxi^{(q)}\bH_i\bfxi^{(q)}}\le \alpha_2, \forall\;i\in \mathcal{S}^{(q)} {\rm and}\ q\in \mathcal{Q}\right).\nonumber
\end{align}
Now notice that for all $i\in \mathcal{M}$ and $q\in \mathcal{Q}$, if $i \in \mathcal{S}^{(q)}$, then it can be easily checked that $q\in \mathcal{P}_i$.
It follows that, by using Lemma \ref{claimXLowerBound} we have that
\begin{align}
\trace\left[\bH_i\hat\bX^{(q)}\right]\ge \hat{\alpha}_i^{(q)} \ge \frac{1}{Q-P_i+1},\ {\rm when}~ i \in \mathcal{S}^{(q)} ~{\rm and}~ q\in \mathcal{Q}.
\end{align}
Then, we have
\begin{align}
&\quad\Prob\left(v^{\rm
P2} \le \bar\mu\  v^{\rm
SDP2} \right)\nonumber\\
&\ge\Prob\left[\sum_{q=1}^{Q}\|\bfxi^{(q)}\|^2\le
\frac{\bar\mu}{\alpha_2}\left(\sum_{q=1}^{Q}\trace[\bhX^{(q)}]\right), \right. \nonumber\\
& \qquad \qquad\left. \bfxi^{(q)}\bH_i\bfxi^{(q)}\ge \frac{Q-P_i+1}{\alpha_2} \trace [\bH_i\hat\bX^{(q)}], \forall \;i\in \mathcal{S}^{(q)} {\rm and}\ q\in \mathcal{Q}\right]\nonumber\\
& \ge 1-\Prob\left(\sum_{q=1}^{Q}\|\bfxi^{(q)}\|^2>
\frac{\bar\mu}{\alpha_2}\left(\sum_{q=1}^{Q}\trace[\bhX^{(q)}]\right)\right)\nonumber\\
&\quad\quad -\sum_{q=1}^{Q}\sum_{i\in \mathcal{S}^{(q)}} \Prob\left( \bfxi^{(q)H}\bH_i\bfxi^{(q)}< \frac{Q-P_i+1}{\alpha_2} \trace [\bH_i\hat\bX^{(q)}]\right)\nonumber\\
&= 1-\Prob\left(\sum_{q=1}^{Q}\|\bfxi^{(q)}\|^2>
\frac{\bar\mu}{\alpha_2}\left(\sum_{q=1}^{Q}\trace[\bhX^{(q)}]\right)\right)\nonumber\\
&\quad\quad -\sum_{i=1}^{M}\sum_{q\in \mathcal{P}_i} \Prob\left( \bfxi^{(q)H}\bH_i\bfxi^{(q)}< \frac{Q-P_i+1}{\alpha_2} \trace [\bH_i\hat\bX^{(q)}]\right)\nonumber\\
&\ge 1-\frac{\alpha_2}{\bar\mu}-\sum_{i=1}^{M}\sum_{q\in \mathcal{P}_i} \Prob\left( \bfxi^{(q)H}\bH_i\bfxi^{(q)}< \frac{Q-P_i+1}{\alpha_2} \trace [\bH_i\hat\bX^{(q)}]\right),
\end{align}
where the last inequality is from the Markov's inequality.

\noindent \textbf{ The Real Case.} When $\mathbb{F}=\mathbb{R}$, by \cite[Lemma 1]{Luo07approximationbounds}, we have that
\begin{align}
&\quad\Prob\left(v^{\rm
P2} \le \bar\mu\  v^{\rm
SDP2} \right)\ge 1-\frac{\alpha_2}{\bar\mu}-\sum_{i=1}^{M}P_i\cdot \delta_i(\alpha_2),
\end{align}
with
$$\delta_i(\alpha_2)=\max \left\{\sqrt{\frac{Q-P_i+1}{\alpha_2}}, \frac{2(r_i-1)}{\pi-2}\cdot \frac{Q-P_i+1}{\alpha_2}\right\},$$
and $r_i=\min\{\rank(\bH_i), \rank(\hat\bX^{(q)})\}$. Since $\bH_i=\bh_i\bh_i^H$, we have $\bar{r}_i=1$ for all $i=1,\cdots, M$ and thus
$\delta_i(\alpha_2)=\sqrt{\frac{Q-P_i+1}{\alpha_2}}$. As a result, we have
\begin{align}
&\quad\Prob\left(v^{\rm
P2} \le \bar\mu\  v^{\rm
SDP2} \right)\ge 1-\frac{\alpha_2}{\bar\mu}-\sum_{i=1}^{M}P_i\cdot\sqrt{\frac{Q-P_i+1}{\alpha_2}}.
\end{align}
By setting $\alpha_2=\frac{9(\sum_{i=1}^{M}P_i\sqrt{Q-P_i+1})^2}{\pi}$, $\bar\mu=3\alpha_2$,
and $\sigma_2=\frac{2-\sqrt{\pi}}{3}$,
we have $$\Prob\left(v^{\rm
P2} \le \bar\mu\  v^{\rm
SDP2} \right)\ge \sigma_2>0.$$

\noindent \textbf{ The Complex Case.} When $\mathbb{F}=\mathbb{C}$, by \cite[Lemma 3]{Luo07approximationbounds}, we have that
\begin{align}
&\quad\Prob\left(v^{\rm
P2} \le \bar\mu\  v^{\rm
SDP2} \right)\ge 1-\frac{\alpha_2}{\bar\mu}-\sum_{i=1}^{M}P_i\cdot \bar\delta_i(\alpha_2),
\end{align}
with
$$\bar\delta_i(\alpha_2)=\max \left\{\frac{4}{3}\cdot\frac{Q-P_i+1}{\alpha_2}, 16(r_i-1)^2\cdot \frac{(Q-P_i+1)^2}{\alpha_2^2}\right\},$$
and $r_i=\min\{\rank(\bH_i), \rank(\hat\bX^{(q)})\}$. Since $\bH_i=\bh_i\bh_i^H$, we have $\bar{r}_i=1$ for all $i=1,\cdots, M$ and thus
$\bar\delta_i(\alpha_2)=\frac{4}{3}\cdot\frac{Q-P_i+1}{\alpha_2}$. As a result, we have
\begin{align}
&\quad\Prob\left(v^{\rm
P2} \le \bar\mu\  v^{\rm
SDP2} \right)\ge 1-\frac{\alpha_2}{\bar\mu}-\frac{4}{3}\sum_{i=1}^{M}P_i\cdot \frac{Q-P_i+1}{\alpha_2}.
\end{align}
By setting $\alpha_2=\frac{4\sum_{i=1}^{M}P_i(Q-P_i+1)}{\sqrt{\pi}}$, $\bar\mu=3\alpha_2$,
and $\sigma_2=\frac{2-\sqrt{\pi}}{3}$,
we have $$\Prob\left(v^{\rm
P2} \le \bar\mu\  v^{\rm
SDP2} \right)\ge \sigma_2>0.$$
This completes the proof.
\end{proof}

\begin{theorem} \label{thm:2.5}
For \eqref{problemQReformulate} and its SDP relaxation \eqref{problemSDR2}, we have
\begin{align}
v^{\rm P2}\le\bar\mu v^{\rm SDP2},
\end{align}
with $\bar\mu=\frac{27(\sum_{i=1}^{M}P_i\sqrt{Q-P_i+1})^2}{\pi}$ when  $\mathbb{F}=\mathbb{R}$ and $\bar\mu=\frac{12\sum_{i=1}^{M}P_i(Q-P_i+1)}{\sqrt{\pi}}$ when  $\mathbb{F}=\mathbb{C}$.
\end{theorem}

The proof of Theorem \ref{thm:2.5} is similar to that of Theorem \ref{thm:2.2}, and we omit it.

It is important to mention here that, $P_i=1$ for all $i$ and $Q=1$ in the real case, we have
\[\bar\mu=\frac{27M^2}{\pi},\]
which corresponding to Luo et al.'s result \cite{Luo07approximationbounds}.
Moreover, when we consider case that $Q=2$ and $P_i=1$ for all
$i$, the proposed algorithm and the approximation analysis reduces to the one that we have discussed
in Section \ref{secTwoSlots}. Furthermore, if we set $M=1$, $P_1=q$ and $Q=m$ in the real case, we have
\[\bar\mu=\frac{27q^2(m-q+1)}{\pi},\]
which is exactly the results that we have obtained for the minimization model with $\epsilon=0$ as in \cite{Xu2013}.

 \section{Numerical Experiments}

 \begin{figure}[t]
\caption{Upper bound on $v^{\min}_{{\rm QP}}/v^{\min}_{{\rm SDP}}$
for $M=5$, $N=4$, $300$ realizations of real Gaussian i.i.d.
$\bh_i$ for $i=1,\cdots, M$ in the real case.} \label{fig_real1}
\includegraphics[width=\textwidth]{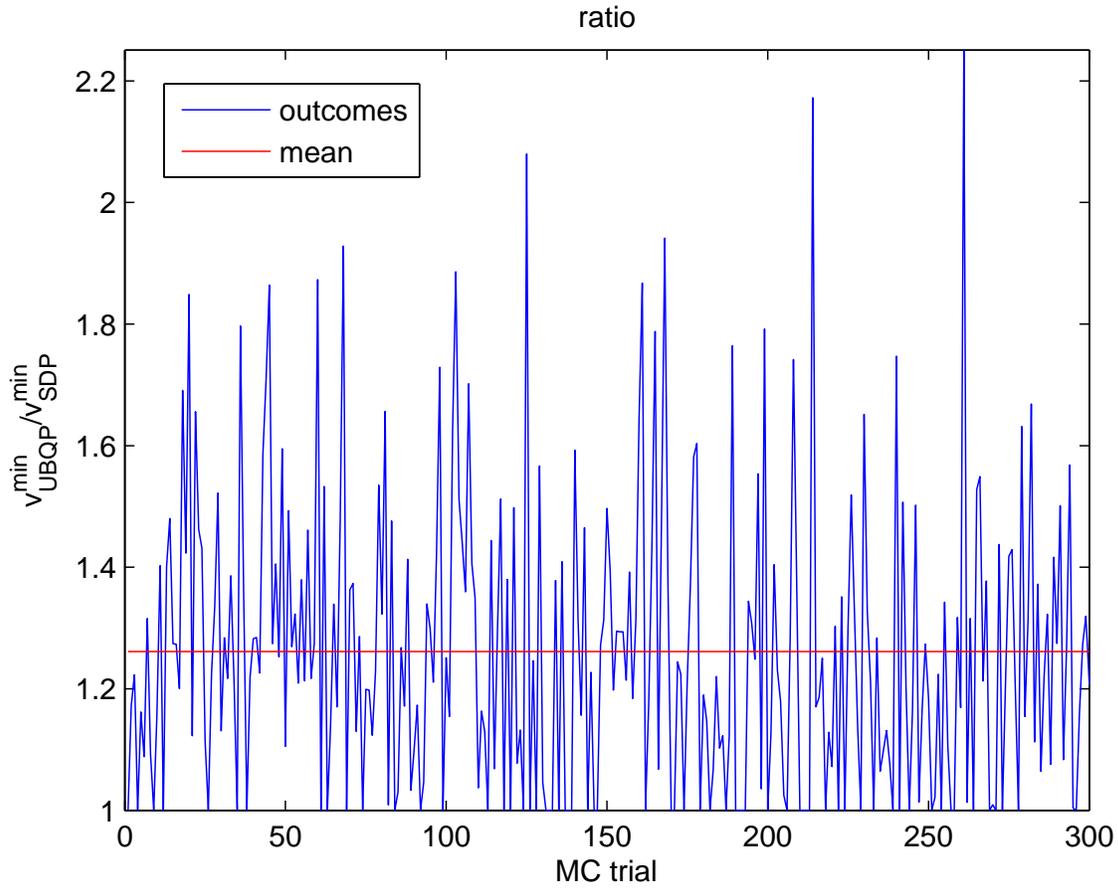}
\end{figure}
\begin{figure}[t]
\caption{Histogram of the outcomes in Figure \ref{fig_real1}.}
\label{fig_real2}
\includegraphics[width=\textwidth]{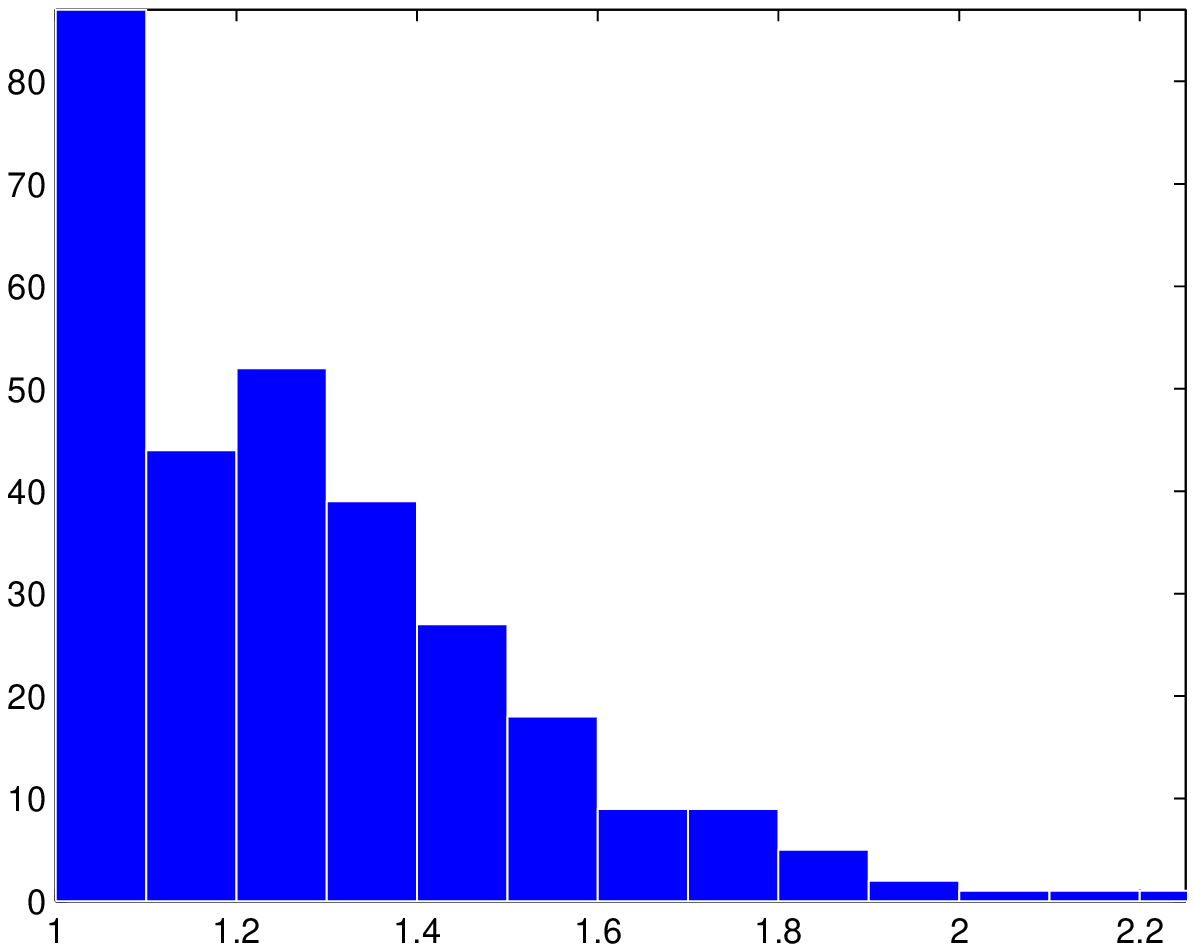}
\end{figure}
\begin{table}
\caption{Mean and standard deviation of the approximation ratio over
$300$ independent realizations of real Gaussian i.i.d. $\bh_i$
$(i=1,\cdots, M)$, when $\mathbb{F}=\mathbb{R}$.}
\begin{center} \footnotesize
\begin{tabular}{|l|l|cccc|} \hline
\multirow{2}{*}{$M$}  & \multirow{2}{*}{$N$} & \multirow{2}{*}{min} &\multirow{2}{*}{max} & \multirow{2}{*}{mean} & \multirow{2}{*}{Std}  \\
&&&&&\\\hline
\multirow{3}{*}{$M=5$} &  $N=4$ & 1.0000 & 2.2508& 1.2612& 0.0590 \\
& $N=6$ & 1.0000 & 4.1582 & 1.2938& 0.1467 \\
& $N=8$ & 1.0000 & 2.4780 & 1.3045& 0.0985\\ \hline
\multirow{3}{*}{$M=10$} &  $N=4$ & 1.0000 & 5.5210& 1.7504& 0.4253 \\
& $N=6$ & 1.0000 & 5.3879 & 1.7640& 0.3313\\
& $N=8$ & 1.0000  & 6.9469 & 2.1074& 0.8202\\ \hline
\multirow{3}{*}{$M=15$} &  $N=4$ & 1.0000 & 9.2963& 2.2168& 1.0545 \\
& $N=6$ & 1.0000& 6.7141& 2.6144& 1.1682\\
& $N=8$ & 1.3109& 6.5277& 2.9619& 1.1806\\ \hline
\end{tabular}
\end{center}
\label{min_real}
\end{table}
\begin{table}
\caption{Mean and standard deviation of upper bound ratio over $300$ independent realizations of real Gaussian i.i.d. $\bh_i$ $(i=1,\cdots, M)$, when $\mathbb{F}=\mathbb{C}$.}
\begin{center} \footnotesize
\begin{tabular}{|l|l|cccc|} \hline
\multirow{2}{*}{$M$}  & \multirow{2}{*}{$N$} & \multirow{2}{*}{min} &\multirow{2}{*}{max} & \multirow{2}{*}{mean} & \multirow{2}{*}{Std}  \\
&&&&&\\\hline
\multirow{3}{*}{$M=5$} &  $N=4$ & 1.0000 & 1.8518 & 1.1336 & 0.0445 \\
& $N=6$ & 1.0000 & 1.9089& 1.0870 & 0.0341 \\
& $N=8$ & 1.0000 & 1.9620 & 1.0915 & 0.0400\\ \hline
\multirow{3}{*}{$M=10$} &  $N=4$ & 1.0000 & 2.2457 & 1.5281& 0.0654 \\
& $N=6$ & 1.0000  & 2.2511& 1.5912& 0.0781\\
& $N=8$ & 1.0000 & 2.3561 &1.5885 & 0.0828\\ \hline
\multirow{3}{*}{$M=15$} &  $N=4$ & 1.0097 & 2.8260 & 1.7460 & 0.0699 \\
& $N=6$ &  1.0001 & 3.0243 & 1.8821 & 0.1064\\
& $N=8$ & 1.0741 & 3.4917 & 1.9442 & 0.1193\\ \hline
\end{tabular}
\end{center}
\label{min_complex}
\end{table}
\begin{figure}[t]
\caption{Upper bound on $v^{\min}_{{\rm QP}}/v^{\min}_{{\rm SDP}}$
for $M=5$, $N=4$, $300$ realizations of real Gaussian i.i.d.
$\bh_i$ for $i=1,\cdots, M$ in the complex case.}
\label{fig_complex1}
\includegraphics[width=\textwidth]{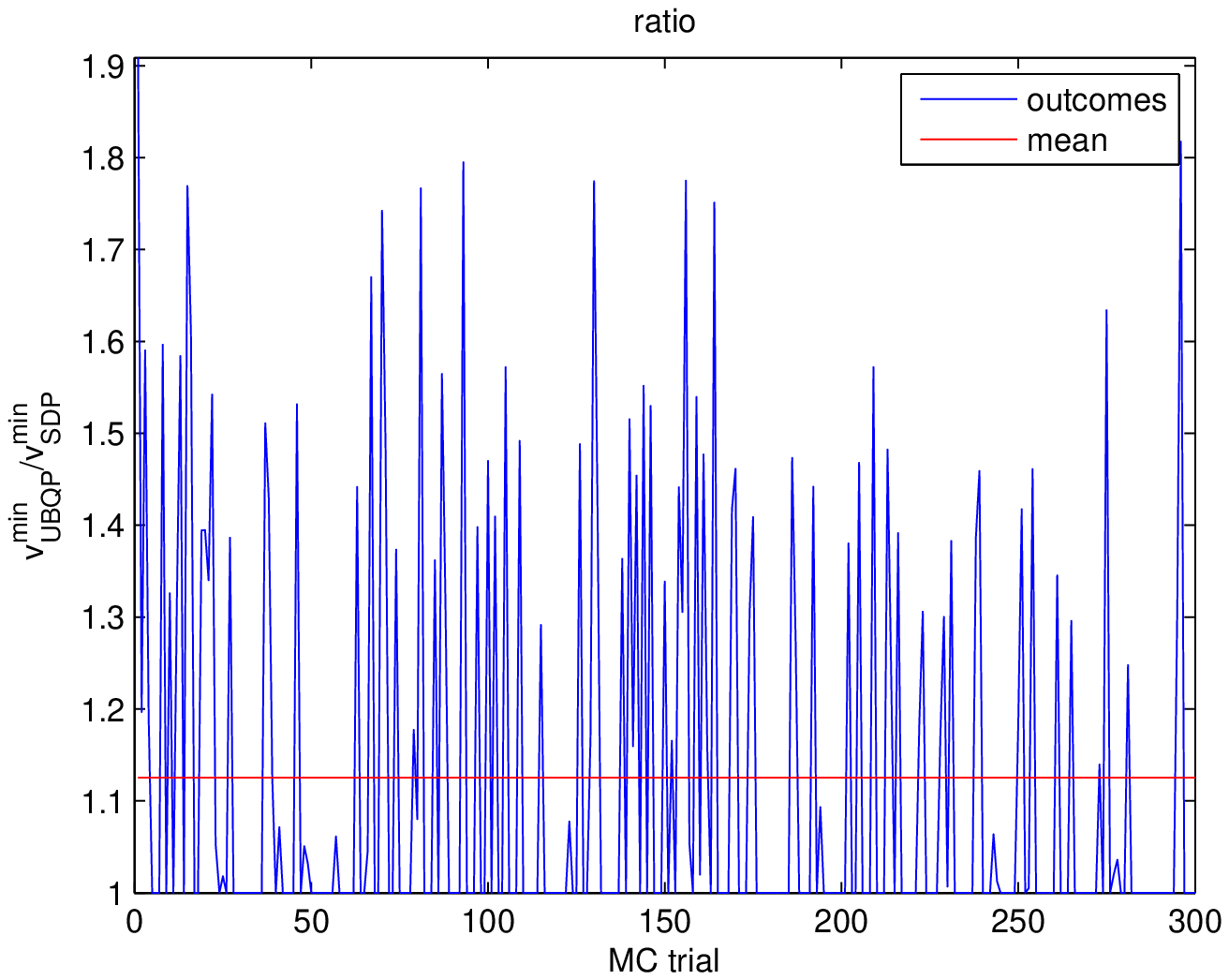}
\end{figure}
\begin{figure}[t]
\caption{Histogram of the outcomes in Figure \ref{fig_complex1}.}
\label{fig_complex2}
\includegraphics[width=\textwidth]{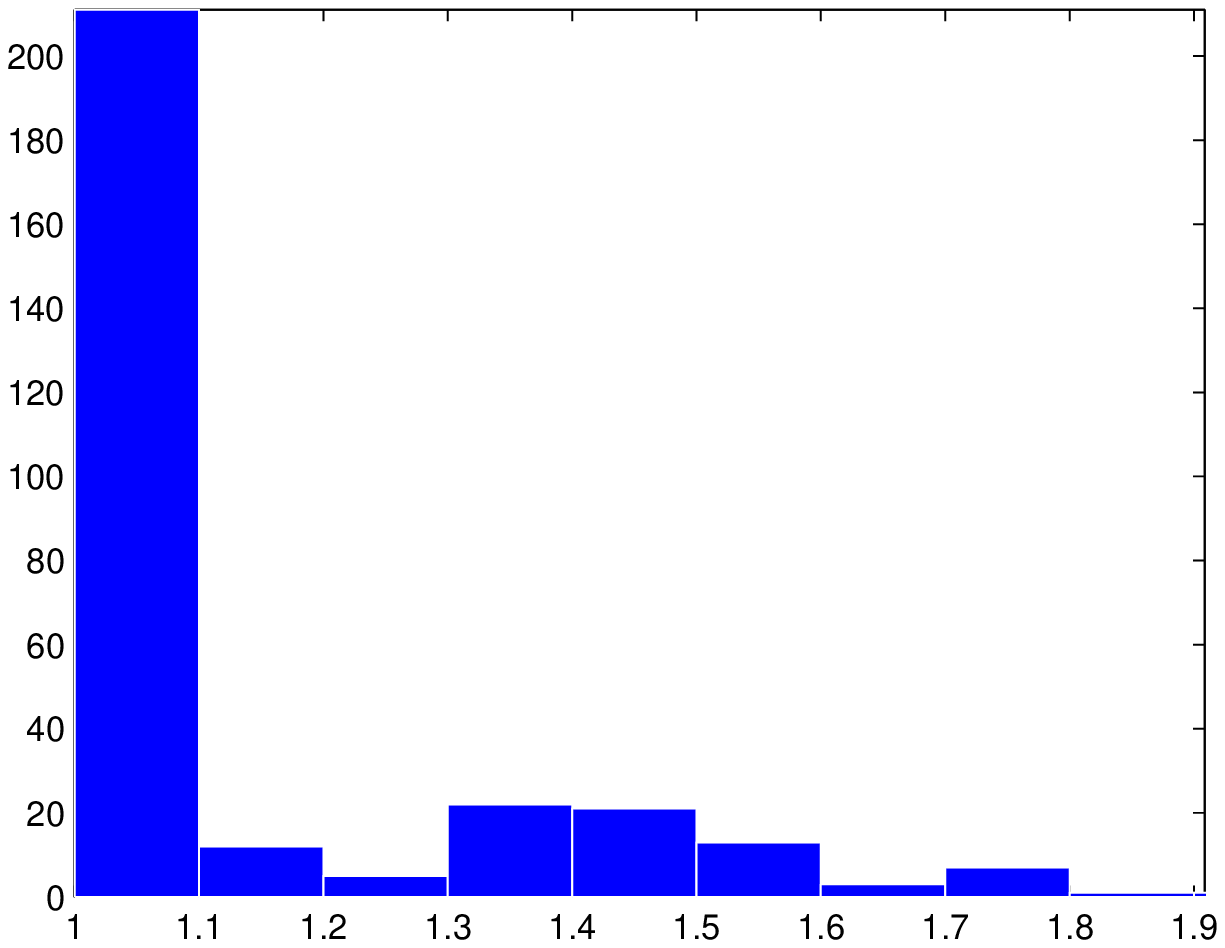}
\end{figure}

In this section we perform numerical study for the proposed algorithms. Throughout this section, we generate
the data matrix $\bH_i$ by using $\bH_i=\bh_i\bh_i^H$ $(i=1,\cdots, M)$, with randomly generated vectors $\bh_i$. The SDP relaxation problems are all solved by CVX \cite{CVX2011}, and the optimal objective value for the SDPs are denoted by $v^{\min}_{{\rm SDP}}$.

We test the proposed procedure listed
in Table \ref{tableRandomize} for \eqref{problem} with different choices of $M$ and
$N$. The Step S3 and Step S4 are repeated by $T=1000$ independent
trials, and the solutions generated by $k$th trial are denoted by $(\bar\bx^{(1)})^k$ and
$(\bar\bx^{(2)})^k$. Let
$$v^{\min}_{{\rm UBQP}}:=\min_{k=1,\cdots, T}\|(\bar\bx^{(1)})^k\|^2+\|(\bar\bx^{(2)})^k\|^2.$$
It is clearly that  $v^{\min}_{{\rm UBQP}}\geq v^{{\rm P1}}$, as a result, $v^{\min}_{{\rm
UBQP}}/v^{\min}_{{\rm SDP}}$ is an upper bound of the
true approximation ratio (which is difficult to obtain in polynomial time).

Table \ref{min_real} shows the average ratio (mean) of
$v^{\min}_{{\rm UBQP}}/v^{\min}_{{\rm SDP}}$ over 300 independent
realizations of i.i.d. real-valued Gaussian $\bh_i$, $(i=1,\cdots,
M)$ for several combinations of $M$ and $N$. The minimum value (min), the  maximum value
(max), the average value (mean) and the standard deviation (Std) of $v^{\min}_{{\rm
UBQP}}/v^{\min}_{{\rm SDP}}$ over 300 independent realizations are
also shown in Table \ref{min_real}. Table \ref{min_complex} shows
the corresponding minimun value, maximum value, average value, and the standard
deviation of $v^{\min}_{{\rm UBQP}}/v^{\min}_{{\rm SDP}}$ for
$\mathbb{F}=\mathbb{C}$. We can see the results in both tables are significantly better than
what is predicted by our worst-case analysis. In all test examples,
the average values of $v^{\min}_{{\rm UBQP}}/v^{\min}_{{\rm SDP}}$
are lower than $3$ (resp. lower than $2$) when
$\mathbb{F}=\mathbb{R}$ (resp. when $\mathbb{F}=\mathbb{C}$).
Moreover, the minimum value of $v^{\min}_{{\rm UBQP}}/v^{\min}_{{\rm SDP}}$ are exactly equal to $1$
in most cases in both real and complex cases, which means the optimal solutions are obtained by our algorithm for some cases.

Figure \ref{fig_real1} plots $v^{\min}_{{\rm UBQP}}/v^{\min}_{{\rm SDP}}$
for 300 independent realizations of i.i.d. real valued Gaussian
$\bh_i$ ($i=1,\cdots, M$) for $M=5$ and $N=4$. Figure
\ref{fig_real2} shows the corresponding histogram. Figure
\ref{fig_complex1} and Figure \ref{fig_complex2} show the
corresponding results for i.i.d complex-valued circular Gaussian
$\bh_i$ ($i=1,\cdots,M$). Both the mean and the maximum of the upper
bound $v^{\min}_{{\rm UBQP}}/v^{\min}_{{\rm SDP}}$ are lower in the
complex case.

Moreover, our numerical results also corroborates well with our
theoretic analysis. First, the upper bound of the approximation
ratio is independent of the dimension of $\bw$: the results vary
only slightly for $N=4$, $N=6$ and $N=8$ in both real and complex case.
Second, from Table \ref{min_real} and Table \ref{min_complex}, it can be shown that for fixed
$N$, the maximum value and the average value of $v^{\min}_{{\rm UBQP}}/v^{\min}_{{\rm
SDP}}$ over 300 independent trials grow as $M$ increases in all
test examples except the case that $M=10$ and $N=8$ in Table \ref{min_real}. It corresponds to the result
in Theorem \ref{thm:2.1}.

 \section{Conclusion and Discussion}
 In this paper, two classes of nonconvex quadratic optimization problems with mixed binary and continuous variables are considered, both of which are motivated by important applications in wireless networks.
New SDP relaxation techniques with novel randomization algorithms are provided for these problems. It is shown that these efficient techniques can provide high quality approximate solutions. Our theoretic analysis provides useful insights on the effectiveness of the new SDP relaxation techniques for these two classes of MBQCQP problems. For both problems, finite and data independent approximation ratios are guaranteed. It should be pointed out that our worst-case analysis of SDP relaxation performance
is based on certain special structure of the discrete and continuous variables. Using a counter example, we have shown in \cite{Xu2013} that SDP relaxation techniques and the corresponding analysis cannot extend directly to general MBQCQP problems.

\ACKNOWLEDGMENT{This research is supported by
the Chinese NSF under the grant 11101261 and the First-class Discipline of Universities in Shanghai.}





\end{document}